\documentclass[11pt,halfline,a4paper]{article}

\usepackage[utf8]{inputenc}
\usepackage[english]{babel}
\usepackage{amsmath}
\usepackage{amsthm}
\usepackage{amsfonts}
\usepackage{amssymb,bm}
\usepackage[pdftex]{color,graphicx}
\usepackage{mathrsfs}  
\usepackage{mathbbol}
\usepackage[expansion]{microtype}
\usepackage{esint}

\usepackage[pdftex,colorlinks=true]{hyperref}
\usepackage[font={small,sf}, labelfont={sf,bf}, margin=1cm]{caption}

\newtheorem{prop}{Proposition}

\newtheorem{theorem}{Theorem}
\newtheorem{corollary}{Corollary}
\newtheorem{remark}{Remark}

\newtheorem*{theorem*}{Theorem}

\newcommand\PP{ \mathbf{P} }

\DeclareMathOperator{\e}{e}

\title{A solvable class of renewal processes}
 \author{Nathana\"el Enriquez and Nathan Noiry}
\date{}

\begin{document}
\maketitle

\abstract{When the distribution of the inter-arrival times of a renewal process is a mixture of geometric laws, we prove that the renewal function of the process is given by the moments of a probability measure which is explicitly related to the mixture distribution. We also present an analogous result in the continuous case when the inter-arrival law is a mixture of exponential laws. We then observe that the above discrete class of renewal processes provides a solvable family of random polymers. Namely, we obtain an exact representation of the partition function of polymers pinned at sites of the aforementioned renewal processes. In the particular case where the mixture measure is a generalized Arcsine law, the computations can be explicitly handled. 
\bigskip

\bigskip

}

\section{Renewal theory for mixtures of geometric laws}

Let $K$ be a probability measure on $\mathbf{N} \cup \{\infty\} = \{1,2, \ldots, \} \cup \{ \infty\}$. Let $\eta = (\eta_n)_{n \geq 1}$ be an i.i.d. sequence of random variables with law $K$. We consider $\eta$ as inter-arrival times of a renewal process and we define the random variables $\tau_0=0$ and for all $n \geq 1$, $\tau_n = \sum_{1 \leq i \leq n} \eta_i$. Finally, we define the renewal process associated to $K$ as the random set $\tau:=\{\tau_i; \, i \geq 0\}$. We will denote by $\mathbf{P}$ the law of $\tau$. 

Before stating our main result, let us introduce the Stieltjes transform of a positive measure, which turns out to be the key notion in this framework. Let $\mathbf{C}_+ := \{ z \in \mathbf{C}, \, \Im z > 0 \}$. Let $\mu$ be a positive measure on $\mathbf{R}$. The Stieltjes transform of $\mu$ is the analytic function $s_\mu$, defined from $\mathbf{C}_+$ to $\mathbf{C}_+$ by:
\begin{equation*}
s_\mu(z) = \int_\mathbf{R} \frac{\mathrm{d}\mu(x)}{x-z}.
\end{equation*}

Let us mention the useful fact that, for almost every real point $x \in \mathbf R$, $s_\mu(x+it)$ converges as $t \rightarrow 0^+$ (see (1.2) in \cite{poltoratski2010hilbert}). Moreover, the real and imaginary parts of the limit correspond respectively to the Hilbert transform of $\mu$ at $x$, denoted by $H_\mu(x)$, and to the density $f_\mu$ of the absolutely continuous part of $\mu$ with respect to the Lebesgue measure (see formula (1.2.8) of Simon \cite{simon2005orthogonal}). We also remind that the Hilbert transform coincides with the Cauchy principal value $H_\mu(x) = \fint \frac{\mathrm{d}\mu(y)}{y-x}$. In short, for almost every $x \in \mathbf{R}$,
\[  \lim\limits_{ t \rightarrow 0^+} s_\mu(x+it) = H_\mu(x) + i \pi f_\mu(x).  \]
\begin{theorem}\label{th:Renewal}
Let $\mu$ be a probability measure on $[0,1]$ and suppose that, for all $n \geq 1$, 
\begin{equation}\label{eq:Assumption0}
\begin{array}{l}
K(n) = \int_0^1 (1-x)^{n-1} x \mathrm{d}\mu(x),\\
K(\{\infty\}) = \mu(\{0\}).
\end{array}
\end{equation}
Then, there exists a probability measure $\nu$ on $[0,1]$ such that:
\begin{equation}\label{eq:Subordination0}
s_\nu(z) s_\mu(1-z) = \frac{1}{z(1-z)}.
\end{equation}
Moreover, for all $N \geq 0$:
\begin{equation}\label{eq:MomentRepresentation0}
\mathbf{P}(N \in \tau) = \int_0^1 x^N \mathrm{d}\nu(x).
\end{equation}
\end{theorem}

\begin{remark}
The mapping from $\mu$ to $\nu$ is an involution. In particular, if one assumes \eqref{eq:MomentRepresentation0}, then the inter-arrival distribution satisfies \eqref{eq:Assumption0} with $\mu$ defined by
\[ s_\mu(z) =  \left[  z (1-z) s_\nu(1-z) \right]^{-1}. \]
This presents some interest for the applications. Indeed, one usually has access to the renewal measure $\mathbf{P}(N \in \tau)$ and would like to infer the underlying inter-arrival distribution.
\end{remark}
\begin{remark}
The family of generalized Arcsine laws with parameters $(1-v, v)$, $v \in (0,1)$, defined by the densities
\begin{equation}
\frac{\sin(\pi v)}{\pi}x^{-v} (1-x)^{v-1}  \mathbf{1}_{x \in [0,1]} \mathrm{d}x,
\end{equation}
are fixed points of the involution $\mu \mapsto \nu$. It can be easily checked from the expression of their Stieltjes transform, which are equal to
\[  \frac{1}{1-z}\left(  \frac{z}{1-z} \right)^{-v} . \]
We conjecture that these distributions are actually the {\it only} fixed points. This family of measures will be further investigated in Section \ref{sec:computations} in the context of random polymers.
\end{remark}

Before proving this Theorem, we state as a Corollary some more explicit formulas for $\mathbf{P}(N \in \tau)$ in two specific but generic cases. 

The first one deals with the case where $\mu$ is a finite sum of Dirac masses. This case was already considered in the frame of the Fixman-Freire algorithm
\cite{Fixman}. Indeed, these authors wanted to use an approximation of $\mu$ by a finite sum of Dirac masses in order to compute an approximation of $\mathbf{P}(N \in\tau)$. It turns out that they compute this last quantity using a rather heavy recursive scheme and we give here a clear and tractable formula for it.

The second one deals with the cases where $\mu$ admits a density with respect to the Lebesgue measure, supported on a finite union of intervals. This includes the previous case considered by Nagaev \cite{Nagaev2015} which assumed among other technical hypothesis that the density of $\mu$ is Lipschitz. The formula we give below for the density of $\nu$ is quite synthetic and seems to amend the intricate formula of Nagaev.
 
\begin{corollary}\label{coro:coro1}
\begin{enumerate}
\item When $\mu$ is a finite sum of Dirac masses, this is also the case for $\nu$. More precisely, let $0 < x_1 < x_2 < \cdots < x_n < 1$ and $a_1, \ldots, a_n >0$ with $a_1 + \cdots + a_n = 1$ be such that $\mu = \sum_{1 \leq i \leq n} a_i \delta_{x_i}$. Then, 
\[ s_\nu(z) = \frac{1}{z(1-z) \sum_{1 \leq i \leq n} \frac{a_i}{z-(1-x_i)}} \]
and $\nu$ is purely atomic and admits exactly $n+1$ atoms which are located in increasing order at $0, y_1, \ldots, y_{n-1},1$ where $y_i$ is the only root of $x \mapsto \sum_{ 1 \leq i \leq n} \frac{a_i}{z - (1-x_i)}$ on the interval $(1-x_{n-i+1},1-x_{n-i})$. Moreover, the mass $\nu(\{y_i\})$ is given by the residue of the rational function $s_\nu$ at $y_i$. In that case, Equation \eqref{eq:MomentRepresentation0} rewrites 
\[  \mathbf{P}(N \in \tau) = \frac{1}{m_K} +  \sum\limits_{1 \leq i \leq n-1} y_i^N \nu(\{y_i\}). \]

\item When $\mu$ is absolutely continuous with respect to the Lebesgue measure with density $f_\mu$ supported on a finite union of disjoint intervals $[a_1,b_1] \sqcup [a_2,b_2] \sqcup \cdots \sqcup [a_n,b_n]$ with $0 \leq a_1 < b_1 < a_2 < \cdots < a_n < b_n \leq 1$, the measure $\nu$ is a sum of a pure point measure $\nu_{pp}$ and an absolutely continuous one $\nu_{ac}$. 

The measure $\nu_{pp}$ admits exactly one atom on each interval $(1-a_{n-i+1},1-b_{n-i})$ located at the point $y_i$ defined as the only root on that interval of the function $x \mapsto \int_0^1 \frac{\mathrm{d}\mu(s)}{s-(1-x)}$. The mass $\nu(\{y_i\})$ is given by the residue of the function $s_\nu$ at $y_i$. Moreover, if $m_K < \infty$, then $\nu$ admits an extra atom at $1$ with mass $1/ m_K$, and if $\int_0^1 \frac{\mathrm{d}\mu(x)}{1-x} < \infty $, then $\nu$ admits also an extra atom at $0$ with mass $ ( \int_0^1 \frac{\mathrm{d}\mu(x)}{1-x} )^{-1}$. Apart from these points, $\nu_{pp}$ admits no other atom.

The support of the measure $\nu_{ac}$ is $[1-b_n,1-a_n] \sqcup \cdots \sqcup [1-b_1,1-a_1]$, and its density is given by:
\begin{align*}  
f_{\nu_{ac}}(x) 
&= \lim\limits_{t \rightarrow 0^+} \frac{1}{\pi} \Im s_{\nu}(x+it)   \\
&=  \frac{1}{x(1-x)} \frac{f_{\mu}(1-x)}{H_\mu(1-x)^2 + \pi^2 f_{\mu}(1-x)^2}.
\end{align*}
In that case, Equation \eqref{eq:MomentRepresentation0} rewrites 
\[  \mathbf{P}(N \in \tau) = \frac{1}{m_K} +  \sum\limits_{1 \leq i \leq n-1} y_i^N \nu(\{y_i\}) + \int_0^1 \frac{(1-x)^{N-1}}{x} \frac{\mathrm{d}\mu(x)}{H_\mu(x)^2 + \pi^2 f_{\mu}(x)^2} . \]
\end{enumerate}
\end{corollary}

\begin{proof}[Proof of Theorem \ref{th:Renewal}.]
Let $G(z) = \sum_{ N \geq 0} \mathbf{P}(N \in \tau) z^N$ be the generating series associated to the sequence $(\mathbf{P}(N \in \tau))_{N \geq 0}$. Then,
\begin{align*}
G(z) &= 1 + E_{\mathbf{P}}\left[ \sum\limits_{k \geq 1} z^{\eta_1 + \cdots + \eta_k} \right] \\
     &= \frac{1}{1- E_{\mathbf{P}}[z^{\tau_1}]}.
\end{align*}
Therefore, using \eqref{eq:Assumption0}, we deduce that:
\begin{equation}\label{eq:EqGen}
-\frac{1}{z}G \left( \frac{1}{z} \right) =   \frac{1}{- z -   \int_0^1 \frac{zx}{1-z-x} \mathrm{d}\mu(x)}.
\end{equation}
The above equality holds for all complex $z \in \mathbf{C}$ such that $1/z$ is inside the disk of convergence of $G$. It extends analytically to the whole complex upper half-plane $\mathbf{C}_+$ thanks to the right-hand side expression. Now, notice that for all $x \in (0,1)$, the homographic function $z \mapsto \frac{zx}{1-z-x}$ preserves $\mathbf{C}_+$ since it is of determinant $x(1-x)>0$. Therefore, the function $z \mapsto -\frac{1}{z}G \left( \frac{1}{z} \right)$ preserves $\mathbf{C}_+$ and is a Nevanlinna's function. Moreover, $-\frac{1}{z}G \left( \frac{1}{z} \right) \sim - \frac{1}{z}$ as $|z| \rightarrow + \infty$. Therefore, using the characterization of Nevanlinna functions which are Stieltjes transforms of probability measures \cite{Akhieezr1965} (page 93), there exists a probability measure $\nu$ such that
\begin{equation}\label{eq:StieltRepres}
-\frac{1}{z}G \left( \frac{1}{z} \right) = \int_\mathbf{R} \frac{\mathrm{d}\nu(x)}{x-z}.
\end{equation}
Identifying the $1/z$ coefficients in Equation \eqref{eq:StieltRepres} leads to Equation \eqref{eq:MomentRepresentation0}.

Once we have got the existence of the measure $\nu$, we want to prove it satisfies Equation \eqref{eq:Subordination0}. Indeed, combining \eqref{eq:EqGen} and \eqref{eq:StieltRepres}, we get that
\begin{equation}
\forall z \in \mathbf{C}_+, \quad \left( \int_\mathbf{R} \frac{\mathrm{d}\nu(x)}{x-z}  \right)^{-1} =  - z -   \int_0^1 \frac{zx}{1-z-x} \mathrm{d}\mu(x).
\end{equation}
Finally, let us notice that since the support of $\mu$, $\mathrm{Supp}(\mu)$, is included in $[0,1]$, its Stieltjes transform $s_\mu$ is analytic on $\mathbf{R} \setminus (0,1)$. In turn, by Equation \eqref{eq:Subordination0}, the Stieltjes transform of $\nu$ is analytic on $\mathbf{R} \setminus (0,1)$, which implies that $\mathrm{Supp}(\nu) \subset [0,1]$. 
\end{proof}


Let us comment on assumption \eqref{eq:Assumption0}. In the proof of Theorem \ref{th:Renewal}, we crucially rely on it in order to prove that the function
\[ z \mapsto  -z E_\mathbf{P}  \left[ \left( \frac{1}{z}  \right)^{\tau_1 } \right]  \]
preserves the upper half-plane. In the generic case, this property is not satisfied. 


Let us finally mention that Theorem \ref{th:Renewal} allows to easily recover the classical renewal Theorem in our setting.

\begin{prop}[Basic renewal theorem]\label{prop:renewalth}
Under the setting of Theorem \ref{th:Renewal},
\begin{equation*}
\mathbf{P}(N \in \tau) \underset{  N \rightarrow + \infty}{ \longrightarrow  } \frac{1}{m_K},
\end{equation*}
where $m_K := \sum\limits_{n \geq 1} n K(n)$.
\end{prop}

\begin{proof}
From Equation \eqref{eq:MomentRepresentation0} and the dominated convergence Theorem,
\[  \mathbf{P}(N \in \tau) \underset{  N \rightarrow + \infty}{ \longrightarrow  } \nu(\{1\}).  \]
By the dominated convergence theorem, the mass $\nu(\{1\})$ is equal to the limit of $-(it) \times s_{\nu}(1+it)$ as $t \rightarrow 0^+$. By Equation \eqref{eq:Subordination0}, this is equal to
\[  \frac{1}{s_\mu (0)} = \left( \int_0^1 \frac{\mathrm{d}\mu(x)}{x}  \right)^{-1} = \left( \int_0^1  \sum\limits_{n \geq 1} n (1-x)^{n-1} x \mathrm{d}\mu(x)  \right)^{-1}  = \frac{1}{m_K}.  \]
\end{proof}

\section{The continuous counterpart: mixture of exponential laws}

In this section, we still consider $\eta = (\eta_n)_{n \geq 1}$ a sequence of i.i.d. positive random, but we now suppose that they admit a density denoted by $f_\eta$ supported on $\mathbf{R}_+$. We will denote by $m_\eta$ their mean.

\begin{theorem}\label{th:RenewalContinuous}
Let $\mu$ be a probability measure on $[0,+\infty)$ and suppose that, for all $x>0$, 
\begin{equation}\label{eq:AssumptionContinuous0}
f_\eta(x) = \int_0^{+\infty} s \e^{- s x} \mathrm{d}\mu(s).
\end{equation}
Define by $H(x)$ the intensity of the renewal process with inter-arrivals $(\eta_i)_{i \geq 1}$.

Then, there exists a positive measure $\nu$ on $[0,+\infty)$ such that:
\begin{equation}\label{eq:SubordinationContinuous0}
(1 + s_\nu(z))s_\mu(z) = \frac{-1}{z}.
\end{equation}
Moreover, for all $x>0$:
\begin{equation}\label{eq:MomentRepresentationContinuous0}
H(x) = \int_0^{+\infty} e^{-xs} \mathrm{d}\nu(s).
\end{equation}
\end{theorem}

As in Corollary \ref{coro:coro1}, we present the two special cases where $\mu$ is a finite sum of Dirac masses and where $\mu$ admits a density with respect to the Lebesgue measure supported on a finite union of intervals. Notice that the first case is another presentation of the content of Chapter 5 of the book of Bladt and Nielsen \cite{bladt2017matrix}, where the authors study the renewal process associated to an absorbed Markov chain which is regularly regenerated.

\begin{corollary}
\begin{enumerate}
\item When $\mu$ is a finite sum of Dirac masses, this is also the case for $\nu$. More precisely, let $0 < x_1 < x_2 < \cdots < x_n$ and $a_1, \ldots, a_n >0$ with $a_1 + \cdots + a_n = 1$ be such that $\mu = \sum_{1 \leq i \leq n} a_i \delta_{x_i}$. Then, 
\[ s_\nu(z) = \frac{-1}{z \sum_{1 \leq i \leq n} \frac{a_i}{x_i - z}} -1 \]
and $\nu$ is purely atomic and admits exactly $n$ atoms which are located in increasing order at $0, y_1, \ldots, y_{n-1}$ where $y_i$ is the only root of $x \mapsto \sum_{ 1 \leq i \leq n} \frac{a_i}{x_i - x}$ on the interval $(x_i,x_{i+1})$. Moreover, the mass $\nu(\{y_i\})$ is given by the residue of the rational function $s_\nu$ at $y_i$. In that case, Equation \eqref{eq:MomentRepresentation0} rewrites 
\[  H(x) = \frac{1}{m_{\eta}} +  \sum\limits_{1 \leq i \leq n-1} \e^{-x y_i} \nu(\{y_i\}). \]

\item When $\mu$ is absolutely continuous with respect to the Lebesgue measure with density $f_\mu$ supported on a finite union of disjoint intervals $[a_1,b_1] \sqcup [a_2,b_2] \sqcup \cdots \sqcup [a_n,b_n]$ with $0 \leq a_1 < b_1 < a_2 < \cdots < a_n < b_n $, the measure $\nu$ is a sum of a pure point measure $\nu_{pp}$ and an absolutely continuous one $\nu_{ac}$. 

The measure $\nu_{pp}$ admits exactly one atom on each interval $(b_i,a_{i+1})$, $1 \leq i \leq n-1$, located at the point $y_i$ defined as the only root on that interval of the function $x \mapsto \int_0^1 \frac{\mathrm{d}\mu(s)}{s-x}$ on the interval $(b_i,a_{i+1})$. The mass $\nu(\{y_i\})$ is given by the residue of the function $s_\nu$ at $y_i$. Moreover, if $m_\eta  < + \infty$, then $\nu$ admits also an extra atom at $0$ with mass $ 1 / m_\eta$. Apart from these points, $\nu_{pp}$ admits no other atom.

The support of the measure $\nu_{ac}$ is $[a_1,b_1] \sqcup \cdots \sqcup [a_n,b_n]$, and its density is given by:
\begin{align*}  
f_{\nu_{ac}}(x) 
&= \lim\limits_{t \rightarrow 0^+} \frac{1}{\pi} \Im s_{\nu}(x+it)   \\
&=  \frac{1}{x} \frac{f_{\mu}(x)}{H_\mu(x)^2 + \pi^2 f_{\mu}(x)^2}.
\end{align*}
In that case, Equation \eqref{eq:MomentRepresentation0} rewrites 
\[  H(x) = \frac{1}{m_{\eta}} +  \sum\limits_{1 \leq i \leq n-1} \e^{-x y_i} \nu(\{y_i\}) + \int_0^{+\infty} \frac{\e^{-sx}}{s} \frac{\mathrm{d}\mu(s)}{H_\mu(s)^2 + \pi^2 f_{\mu}(s)^2} . \]
\end{enumerate}
\end{corollary}

\begin{proof}[Proof of Theorem \ref{th:RenewalContinuous}.]
As usual, we introduce the Laplace transform of $H$ and denote it by $\mathcal{L}$ so that
\[  \mathcal{L}(\lambda) = \int_0^{+\infty} \e^{- \lambda x} H(x) \mathrm{d}x.   \]
Notice that this expression is finite for all $\lambda \in \mathbf{C}$ with positive real part.
Since $H(x) = \sum_{k \geq 1} f_{ \eta_1 + \cdots + \eta_k}(x)$, we deduce that:
\[ \forall \lambda \in \mathbf{C} \, \, \text{s.t.} \, \, \Re (\lambda) >0, \quad \mathcal{L}(\lambda) = \sum_{k \geq 1} \mathbf{E}\left[  \e^{- \lambda ( \eta_1 + \cdots + \eta_k )}   \right] = \frac{1}{1 - \mathbf{E}\left[ \e^{ - \lambda \eta_1 } \right]} - 1. \]
From Assumption \eqref{eq:AssumptionContinuous0}, we deduce that 
\[  \mathbf{E}\left[ \e^{ - \lambda \eta_1 } \right] = \int_0^{+\infty} \frac{s}{s+\lambda} \mathrm{d}\mu(s) . \]
Hence, for all $\lambda \in \mathbf{C}$ with positive real part,
\[  \mathcal{L}(\lambda) = \frac{1}{\int_0^{ + \infty } \frac{\lambda}{s + \lambda} \mathrm{d}\mu(s) } -1 = \frac{1}{\lambda s_{\mu}(-\lambda) } -1. \]
Since the Stieltjes transform of $\mu$ is defined on $\mathbf{C} \setminus \mathbf{R}_+$, the function $\mathcal{L}$ can be analytically extended to $\mathbf{C} \setminus \mathbf{R}_+$.
Since for all $s >0$, the homography $\lambda \mapsto \lambda / ( \lambda+s)$ preserves the lower half-plane, the function $ \lambda \mapsto \mathcal{L}(-\lambda)$ preserves the upper half-plane. Moreover, when $|\lambda| \rightarrow + \infty$, 
\[  \mathcal{L}(-\lambda) = \frac{\int_0^{ + \infty}  \frac{s}{s - \lambda} \mathrm{d}\mu(s)}{ \int_0^{ + \infty}  \frac{-\lambda}{s - \lambda} \mathrm{d}\mu(s)} \sim \frac{-1}{\lambda} \int_0^{+\infty} s \mathrm{d}\mu(s).  \] 
Therefore, using the characterization of Nevanlinna functions which are Stieltjes transforms of positive measures \cite{Akhieezr1965} (page 93), there exists a positive measure $\nu$ with total mass $\int_0^{+\infty} s \mathrm{d}\mu(s)$ whenever this integral is finite, such that
\begin{equation}\label{eq:StieltRepresContinuous}
\mathcal{L}(-\lambda) = s_\nu(\lambda).
\end{equation}

\end{proof}

\section{Application to polymers pinned on a defect line}
In this section, we first recall the definition and main results about a classical random polymer model associated to a renewal process, as presented in the book of Giacomin \cite{Giambattista2007}. Then, we present our contribution which consists in an explicit integral representation of the partition functions of the model, when the inter-arrival time distribution of the underlying renewal process is a mixture of Geometric laws.

\subsection{Definition of the model}
Let $\beta \in \mathbf{R}$ and $N \geq 1$. The polymer model associated to $K$ with parameter $\beta$ is defined by the following probability measure $\mathbf{P}_{N,\beta}$ on subsets of $\{0, \ldots, N\}$ whose density with respect to $\mathbf{P}$ is:
\begin{equation*}
\frac{\mathrm{d}\mathbf{P}_{N,\beta}}{\mathrm{d}\mathbf{P}}( \tau ) := \frac{1}{Z_{N,\beta}} \exp\left( \beta \mathcal{N}_N(\tau) \right) \mathbf{1}_{ N \in \tau},
\end{equation*}
where $\mathcal{N}_N(\tau)  = |\{1, \ldots, N\} \cap \tau |$ and
\[ Z_{N,\beta} := E_\mathbf{P} \left[ \exp \left( \beta \mathcal{N}_N(\tau)   \right) \mathbf{1}_{ N \in \tau}  \right].  \]

The renormalization constant $Z_{N,\beta}$ is called the partition function and captures many information on the model. For example, 
\begin{equation*} 
\frac{1}{N} \frac{\partial}{\partial \beta} \log Z_{N, \beta} = E_{\mathbf{P}_{N,\beta}} \left[ \frac{\mathcal{N}_N(\tau)}{N}  \right]  
\end{equation*}
is the average time spent at $0$ by the polymer. As $N$ tends to infinity, it converges to the derivative of the so-called free energy of the model, which is defined by:
\begin{equation}\label{eq:defnFreeEn}
F(\beta) := \lim\limits_{N \rightarrow + \infty} \frac{1}{N} \log Z_{N,\beta}.
\end{equation}
Hence, $F'(\beta)$ corresponds to the asymptotic fraction of time spent at zero by the polymer. Therefore, we speak about a {\it delocalized} regime when $F'(\beta)=0$ and about a {\it localized} regime when $F'(\beta)>0$.

It has been shown that there exists a phase transition for this model. More precisely, let $\beta_c :=  - \log ( 1 - K(\{\infty\}) )$. Then, if $\beta > \beta_c$, the free energy $F(\beta)$ is uniquely determined by
\begin{equation}\label{eq:DefnForF}
E_K \left[  \exp \left( - F(\beta) \tau_1  \right)  \right] = \exp\left(- \beta  \right).
\end{equation}
Otherwise, if $\beta \leq \beta_c$, the free energy is given by $F(\beta)=0$. Therefore, the model exhibits a phase transition at $\beta = \beta_c$ from a delocalized regime to a localized regime. 

The identification of the free energy is based on a simple rewriting of the partition function that we recall here. First, let us introduce a new family of (sub)-probability measures:
\begin{equation}\label{eq:DefnKtilde}
\forall n \geq 1, \quad \widetilde{K}_\beta(n) := 
\left\{
\begin{array}{lr}
\exp(\beta) K(n) \exp(-F(\beta)n) & \text{if $\beta \geq \beta_c $,} \\
\exp(\beta) K(n) & \text{if $\beta < \beta_c$}.
\end{array}
\right.
\end{equation}
Notice that from the definition $\beta_c$, the measure $\widetilde{K}_\beta$ is a probability measure when $\beta > \beta_c$, whereas it is a sub-probability measure when $\beta < \beta_c$. Let $\widetilde{\mathbf{P}}_\beta$ be the law of the renewal process associated to $\widetilde{K}_\beta$. Then, summing over the inter-arrival times leads to:
\begin{align}
Z_{N,\beta}
&= \sum\limits_{n=1}^N \sum\limits_{ l_1 + \cdots + l_n = N} \prod\limits_{i=1}^n \exp(\beta) K(l_i) \notag \\
&= \exp( F(\beta) N ) \sum\limits_{n=1}^N \sum\limits_{l_1 + \ldots + l_n = N} \prod\limits_{i=1}^n \widetilde{K}_\beta(l_i) \notag \\
&= \exp( F(\beta) N ) \widetilde{\mathbf{P}}_\beta(N \in \tau) \label{eq:FormulaFreeEn}.
\end{align}
Under the classical assumption that there exists $\alpha>0$ and a slowly varying function such that
\begin{equation}\label{eq:Assump}
K(n) = \frac{L(n)}{n^{1+\alpha}},
\end{equation}
the probability $\widetilde{\mathbf{P}}_\beta(N \in \tau)$ does not vanish exponentially fast. Therefore, the function $F$ defined in \eqref{eq:DefnForF} is indeed the free energy of the model. Moreover, thanks to the asymptotic theory of renewal processes, Equation \eqref{eq:FormulaFreeEn} also allows to obtain the asymptotic leading term of $Z_{N,\beta}$ as $N \rightarrow + \infty$.

\subsection{Moment representation of the partition function}
When the inter-arrival time distribution of the underlying renewal process is a mixture of Geometric laws, the partition function $Z_{N,\beta}$ is the $N$-th moment of some measure $\nu_\beta$ which corresponds to an explicit transformation of the mixture distribution. 
\begin{theorem}\label{th:MomentRepresentation}
Let $\mu$ be a probability measure on $[0,1]$ and suppose that, for all $n \geq 1$, 
\begin{equation}\label{eq:Assumption}
\begin{array}{l}
K(n) = \int_0^1 (1-x)^{n-1} x \mathrm{d}\mu(x),\\
K(\{\infty\}) = \mu(\{0\}).
\end{array}
\end{equation}
Then, for all $\beta \in \mathbf{R}$, there exists a probability measure $\nu_\beta$ such that:
\begin{equation}\label{eq:Subordination}
s_{\nu_\beta}(z) \left( \e^\beta s_\mu(1-z) - \frac{1 - \e^\beta}{1-z}  \right) = \frac{1}{z(1-z)}.
\end{equation}
Moreover, for all $N \geq 0$:
\begin{equation}\label{eq:PartFunctionRepres}
Z_{N,\beta} = \int_\mathbf{R} x^N \mathrm{d}\nu_\beta(x).
\end{equation}
\end{theorem}
\begin{proof}
Let $\beta \in \mathbf{R}$ and define, for all complex number in the upper half-plane $z \in \mathbf{C}_+$,  $z_\beta = z \exp( F(\beta) )$. Thanks to Equation \eqref{eq:FormulaFreeEn}, the generating function associated to the sequence $(Z_{N,\beta})_{N\geq 0}$ is equal to:
\begin{align} 
G(z) 
:&= \sum_{N \geq 0}Z_{N,\beta} z^N \notag \\
&= 1 + \sum\limits_{N \geq 1} \widetilde{\mathbf{P}}_\beta(N \in \tau) z_\beta^N \notag \\
&= 1 + E_{ \widetilde{\mathbf{P}}_\beta } \left[ \sum\limits_{ N \geq 1} \mathbf{1}_{N \in \tau} z_\beta^N  \right]. \label{eq:GenFunct1}
\end{align} 
Let $\big(\widetilde{\eta}^{(\beta)}_i \big)_{i\geq 1}$ be a sequence of i.i.d. random variables with law $\widetilde{K}_\beta$. Then, Equation \eqref{eq:GenFunct1} becomes:
\begin{equation*}
G(z) = 1 + E_{\widetilde{\PP}_{\beta}} \left[ \sum\limits_{ k \geq 1}  z_\beta^{\widetilde{\eta}^{(\beta)}_1 + \cdots + \widetilde{\eta}^{(\beta)}_k} \right] = \frac{1}{1- \exp(\beta) E_{\widetilde{\PP}_{\beta}} \left[z_\beta^{\widetilde{\eta}^{(\beta)}_1}\right] }.
\end{equation*}
Finally, from the definition of $\widetilde{K}_\beta$ given in \eqref{eq:DefnKtilde}, we obtain:
\begin{equation*}
G(z) = \frac{1}{1 - \exp(\beta) E_{\mathbf{P}} \left[  z^{\tau_1} \right]}.
\end{equation*}
Let $S(z) = -\frac{1}{z}G\left( \frac{1}{z} \right)$. Then, using Hypothesis \eqref{eq:Assumption}, it is easy to deduce that:
\begin{equation} \label{eq:SubordinationAtom}
S(z) = \frac{1}{- z - \exp(\beta) \int_0^1 \frac{zx}{1-z-x} \mathrm{d}\mu(x)}.
\end{equation}
Notice that $S(z)\sim -\frac{1}{z}$ as $|z| \rightarrow +\infty$. Moreover, since for all $z \in \mathbf{C}_+$ and $x \in (0,1)$, $\frac{zx}{1-z-x} \in \mathbf{C}_+$, the following inclusion holds: $S(\mathbf{C}_+) \subset \mathbf{C}_+$. By \cite{Akhieezr1965} (page 93), these two properties imply that there exists a probability measure $\nu_\beta$ such that:
\begin{equation*}
S(z) = \int_\mathbf{R} \frac{\mathrm{d}\nu_\beta(x)}{x-z},
\end{equation*}
which ends the proof of Theorem \ref{th:MomentRepresentation}.
\end{proof}

\begin{remark}\label{remark:atom}
Note that by \eqref{eq:SubordinationAtom}, the Stieltjes transform of $\nu_{\beta}$ is also given by
\[  s_{\nu_\beta}(z) =  \frac{1}{- z - \exp(\beta) \int_0^1 \frac{zx}{1-z-x} \mathrm{d}\mu(x)}. \]
Observe that the function $y \mapsto \int_0^1 \frac{x}{1-y-x} \mathrm{d}\mu(x)$ is increasing on $(1, +\infty)$ and tends to $- \mu((0,1]) = - \sum_{n \geq 1} K(n) = - (1 - K(\{\infty\}))$ as $y \rightarrow 1^+$. Therefore, the denominator of $s_{\nu_\beta}$ has a single root $x_\beta$ on $(1, +\infty)$, which implies that the restriction of $\nu_\beta$ on $(1, +\infty)$ consists in a unique atom at $x_\beta >1$ if $\beta > - \log( 1 - K(\{\infty\}))$ and is null otherwise. By Equation \eqref{eq:PartFunctionRepres}, $F(\beta) = \log(x_\beta)$ if $\beta > - \log( 1 - K(\{\infty\}))$, and $F(\beta)=0$ otherwise.

The mass of the atom of $\nu_\beta$ at $\exp(F(\beta))$ is given by the residue of $s_{\nu_\beta}$ at $\exp ( F(\beta) )$ which happens to be equal to $F'(\beta)$, which is the average time spent at zero by the polymer. In particular, as $N$ tends to infinity,
\[  Z_{N,\beta} \sim  F'(\beta) \exp( N F(\beta) ).  \]
This formula was already proved in a general framework, see for example Chapter 2 of \cite{Giambattista2007}.

%

\end{remark}

\subsection{Correlation length}
In this section, we are interested, for all $b>0$, in the renewal process $\tau(b)$ with inter-arrival distribution defined by
\[   K_b(n) = \frac{1}{c(b)} K(n) \exp(-nb),  \]
and with mean value denoted by $m_{K_b}$. Namely, we want to compute the rate of convergence in the renewal theorem associated to this process. It is motivated by the article of Giacomin \cite{giacomin2008renewal}, which makes the link between this quantity and the correlation decay for the polymer with law $\mathbf{P}_{N,\beta}$, where $\beta>0$ corresponds to the only positive solution of the equation $b=F(\beta)$, in the limit $N \rightarrow +\infty$.

\begin{prop} \label{prop:propcorr}
Suppose that $\mu$ is a probability measure on $[0,1]$ and that 
\begin{equation}
\begin{array}{l}
K(n) = \int_0^1 (1-x)^{n-1} x \mathrm{d}\mu(x),\\
K(\{\infty\}) = \mu(\{0\}).
\end{array}
\end{equation}
Then, for all $b>0$, there exists a probability measure $\nu_b$ on $[0,1]$ such that for all $N \geq 0$,
\begin{equation} \label{eq:propcorr}
 \mathbf{P}(N \in \tau(b) ) = \int_0^1 x^N \mathrm{d}\nu_b(x).
\end{equation}
Moreover, if $\mu$ has a positive density on a neighborhood of $0$ and that
\begin{equation}\label{eq:correlation}
\lim\limits_{ N \rightarrow +\infty} \frac{1}{N} \log \left( \mathbf{P}(N \in \tau(b) ) - \frac{1}{m_{K_b}}  \right) = -b.
\end{equation}
\end{prop} 

\begin{remark}
Notice that Equation \eqref{eq:correlation} is valid for all $b>0$, whereas the general result of Giacomin \cite[Theorem 1.1]{giacomin2008renewal} establishes this formula only for parameters $b>0$ up to some positive and implicit value $b_0>0$.
\end{remark}

\begin{remark}
Suppose that there exists $a \in (0,1)$ such that $\mu([0,a))=0$ and such that $\mu$ has a positive density at $a^+$. Then, the limit \eqref{eq:correlation} is equal to $-b + \log(1-a)$.
\end{remark}

\begin{proof}[Proof of Proposition \ref{prop:propcorr}.]
We proceed like in the proof of Theorem \ref{th:Renewal}. Namely, denoting by $G_b(z) = \sum_{N \geq 0} \mathbf{P}(N \in \tau(b)) z^N$ the generating series of the renewal probabilities, and introducing $\beta>0$ to be the only solution of the equation $b=F(\beta)$, a direct computation leads to:
\[ G_b(z) = \frac{1}{1 - \e^{\beta}z \mathbf{E}\left[ (z\e^{-b})^{\tau_1}  \right]}.  \]
\begin{align*}
-\frac{1}{z}G_b \left( \frac{1}{z} \right) 
&= \frac{-1}{-z - \e^{\beta-b} (\e^bz)  \mathbf{E}\left[ \left(\frac{1}{z\e^{b}}\right)^{\tau_1}  \right]  } \\
&= \frac{-1}{-z - \e^{\beta-b}  \int_0^1 \frac{z e^b x}{1-z \e^b-x} \mathrm{d}\mu(x)  }  \\
&= \frac{-1}{-z - \e^{\beta}  \int_0^1 \frac{z x}{1-z \e^b-x} \mathrm{d}\mu(x)  } \\
&= \frac{1}{-z + \e^\beta z + z \e^\beta (1 - z \e^b ) s_\mu( 1 - z \e^b  )}.
\end{align*}
Using the penultimate equality, one can check that the analytic function $z \mapsto - \frac{1}{z} G_b ( \frac{1}{z} ) $ preserves the upper half-plane. Moreover, $- \frac{1}{z} G_b ( \frac{1}{z} ) \sim - 1 /z$ as $|z| \rightarrow +\infty$. Therefore, by \cite{Akhieezr1965} (page 93), there exists a probability measure $\nu_b$ such that
\begin{equation}  \label{eq:nu_b}
s_{\nu_b}(z) =  \frac{1}{-z + \e^\beta z + z \e^\beta(1 - z \e^b ) s_\mu( 1 - z \e^b  )}.
\end{equation}
This yields Equation \eqref{eq:propcorr}.

Notice that, by a monotonicity argument, the denominator of $s_{\nu_b}$ vanishes only at $z=0$ and $z=1$. Moreover, as in Proposition \eqref{prop:renewalth}, one can prove that $\nu_b(\{1\}) = 1 / m_{K_b}$. Finally, the probability measure $\nu_b$ has a positive density at $x$ if and only if $\lim_{t \rightarrow 0^+} \Im s_{\nu_b}(x+it) >0$, which, by Equation \eqref{eq:nu_b}, is the case if and only if $\lim_{t \rightarrow 0^+} \Im s_{\mu}(1 - (x+it)\e^b) >0$. Since by assumption, the probability measure $\mu$ has a density in a neighborhood of $0^+$, we deduce that $\nu_b$ has a density in a neighborhood of $(\e^{-b})^-$. Moreover, by Equation \eqref{eq:nu_b}, for all $x \in (\e^{-b},1)$:
\begin{itemize}
\item $\lim_{t \rightarrow 0^+} |  s_{\nu_b}(x+it)  | < +\infty  $,
\item $\lim_{t \rightarrow 0^+} \Im ( s_{\nu_b}(x+it) ) = 0 $.
\end{itemize}
From the first fact (see formula (1.2.10) in \cite{simon2005orthogonal}), we deduce that $\nu_b$ has no singular part on $(\e^{-b},1)$. On the other hand, the second fact implies that $\nu_b$ has no absolutely continuous part either. Hence $\nu$ does not charge the interval $(\e^{-b},1)$.
This yields \eqref{eq:correlation}.

\end{proof}

\section{Computations in the case of generalized Arcsine laws}  \label{sec:computations}
It turns out that some particular choices of probability measure $\mu$ in Theorem \ref{th:MomentRepresentation} yield explicit computations. More precisely, for all $v \in (0,1)$, let $\mu_v$ be the Beta distribution with parameters $(1-v, v)$, defined by:
\begin{equation*}
\mathrm{d} \mu_v(x) := \frac{\sin(\pi v)}{\pi}x^{-v} (1-x)^{v-1}  \mathbf{1}_{x \in [0,1]} \mathrm{d}x.
\end{equation*}
Let $K_v$ be the probability measure on $\mathbf{N}$ associated to $\mu_v$, that is: 
\begin{equation*}
\forall n \geq 1, \quad K_v(n) = \int_\mathbf{R} (1-x)^{n-1} x \mathrm{d}\mu_v(x) = \frac{\sin(\pi v)}{\pi} \frac{\Gamma(n+v-1) \Gamma(2-v)}{\Gamma(n+1)}.
\end{equation*}
As $n \rightarrow + \infty$, $K_v(n) \sim \frac{\sin(\pi v) \Gamma(2-v)}{\pi} \frac{1}{n^{2-v}}$. Hence, the probability measures $K_v$'s satisfy \eqref{eq:Assump}.

Denoting $Z_{N,\beta,v}$ the partition function of the polymer associated to $K_v$, Theorem \ref{th:MomentRepresentation} translates into the following result.

\begin{prop}\label{th:ExplicitComp}
For all $\beta \in \mathbf{R}$ and $v \in (0,1)$, define:
\[ f_{v,\beta}(x) = \frac{  \sin(\pi v)}{\pi x} \frac{\e^\beta x^{1-v} (1-x)^{1-v} }{ (1- \e^\beta)^2 x^{2(1-v)} - 2 \e^\beta (1-\e^\beta) \cos(\pi v) x^{1-v} (1-x)^{1-v} + \e^{2\beta}(1-x)^{2(1-v)} }. \]
Moreover, for all $\beta>0$, define:
\[ \gamma_{v,\beta} = (1-\e^{-\beta})^{\frac{1}{1-v}}, \quad \quad  x_{v,\beta} = \frac{1}{1-\gamma_{v,\beta}}  \quad \text{and} \quad  c_{v,\beta} = \frac{\exp(-\beta)}{1-v} \frac{\gamma_{v,\beta}^v}{1-\gamma_{v,\beta}}.  \]
For all $\beta \in \mathbf{R}$, let $\nu_{v,\beta}$ be the following probability measure:
\begin{align*}
\mathrm{d} \nu_{v,\beta} & (x) =  f_{v,\beta}(x) \mathbf{1}_{x \in (0,1)} \mathrm{d}x + \mathbf{1}_{\beta>0} c_{v,\beta} \mathrm{d} \delta_{x_{v,\beta}}(x).
\end{align*}
Then, for all $N \geq 0$,
\begin{equation*}
Z_{N,\beta,v} = \int_\mathbf{R} x^N \mathrm{d}\nu_{v,\beta}(x).
\end{equation*}
\end{prop}
\begin{remark}
Using Remark \ref{remark:atom}, we deduce that the constant $ c_{v,\beta}$ is a positive constant smaller than 1 which is equal to the asymptotic average time spent at 0 as $N$ tends to infinity. 
\end{remark}

\begin{proof}
Let $v \in (0,1)$ and $\beta \in \mathbf{R}$. Recall that from Equation \eqref{eq:Subordination}, the Stieltjes transform of $\nu_{v,\beta}$ satisfies:
\begin{equation}\label{eq:SubordinationExpl}
\forall z \in \mathbf{C}_+, \quad \frac{1}{s_{\nu_{v,\beta}}(z)} = z (1-z) \left( \e^\beta s_{\mu_v}(1-z) - \frac{1 - \e^\beta}{1-z}  \right).
\end{equation}
For all $x \in \mathbf{R}$, let us define $s_{\nu_{v,\beta}}(x) := \lim_{t \rightarrow 0^+} s_{\nu_{v,\beta}}(x+it)$ and $s_{\mu_v}(x) := \lim_{t \rightarrow 0^+} s_{\mu_v}(x-it)$. Then, Equation \eqref{eq:SubordinationExpl} becomes:
\begin{equation}\label{eq:StieltonR}
\frac{1}{s_{\nu_{v,\beta}}(x)} = x(1-x) \left( \e^\beta s_{\mu_v}(1-x) - \frac{1 - \e^\beta}{1-x}  \right).
\end{equation}
We now use the following identity:
\begin{equation}\label{eq:LimRealStielt}
s_{\mu_v}(x) = \fint_\mathbf{R} \frac{\mathrm{d} \widetilde{\mu_v}(y)}{y-x}  - i \pi \frac{ \mathrm{d} \widetilde{\mu_v}(x) }{\mathrm{d}x }(x)  ,
\end{equation}
where the integral in the right-hand side stands for a Cauchy principal value. The latter can be explicitly computed (see \cite{bateman1954tables} page 250). More precisely:
\begin{equation*}
\frac{1}{\pi} \fint_\mathbf{R} \frac{y^{-v}(1-y)^{v-1}  }{y-x} \mathrm{d}y  = 
\left\{
\begin{array}{lr}
     \frac{1}{\sin( \pi (1-v))} \frac{1}{1-x}  \left| \frac{x}{1-x} \right|^{-v}   & \text{if $x<0$ or $x>1$}, \\
    -x^{-v} (1-x)^{v-1} \frac{\cos( \pi (1-v))}{\sin( \pi (1-v))}      & \text{if $0<x<1$}.
\end{array}
\right.
\end{equation*}
Using \eqref{eq:LimRealStielt}, we deduce that:
\begin{equation}\label{eq:Bateman}
s_{\mu_v}(1-x) = 
\left\{
\begin{array}{lr}
   - \frac{1}{x} \left| \frac{1-x}{x} \right|^{-v}  & \text{if $x<0$ or $x>1$}, \\
    \cos( \pi v) (1-x)^{-v} x^{v-1}   - i \sin(\pi v) x^{-v} (1-x)^{v-1}  & \text{if $0<x<1$}.
\end{array}
\right.
\end{equation}
From \eqref{eq:Bateman} and \eqref{eq:StieltonR}, it is possible to identify the measure $\nu_{v,\beta}$ as explained in the following.

First, the absolutely continuous part of $\nu_{v,\beta}$ is given by $\frac{1}{\pi} \Im s_{\nu_{v,\beta}}(x)$. Therefore, it is supported on the interval $(0,1)$ and given by:

\begin{align*}
\frac{\mathrm{d}\nu_{v,\beta}}{\mathrm{d}x}(x) = \frac{  \sin(\pi v)}{\pi x} \frac{\e^\beta x^{1-v} (1-x)^{1-v} }{ (1- \e^\beta)^2 x^{2(1-v)} - 2 \e^\beta (1-\e^\beta)  \cos(\pi v) x^{1-v} (1-x)^{1-v} + \e^{2\beta}(1-x)^{2(1-v)} }.
\end{align*}

Besides, $\nu_{v,\beta}$ has an atom at $x \in \mathbf{R}$ if $s_{\nu_{v,\beta}}(x) = \infty$. Therefore, the atomic part is contained in $\mathbf{R} \setminus [0,1]$ and $x \in \mathbf{R} \setminus[0,1]$ is an atom of $\nu_{v,\beta}$ if and only if:
\[ 1 + \e^\beta \left( \frac{x}{x-1} \right)^{v-1} - \e^\beta = 0 \quad \Longleftrightarrow \quad x = 1 + \frac{(1-\e^{-\beta})^{ \frac{1}{1-v} }}{1 - (1-\e^{-\beta})^{ \frac{1}{1-v} }}. \]
The right-hand side does not belong to $[0,1]$ if and only if $\beta > 0$. Therefore, we have the following dichotomy:
\begin{itemize}
\item if $\beta > 0$, the measure $\nu_{v,\beta}$ has an atom at $x_{v,\beta} := 1 + \frac{(1-\e^{-\beta})^{ \frac{1}{1-v} }}{1 - (1-\e^{-\beta})^{ \frac{1}{1-v} }} > 1$;
\item if $\beta \leq 0$, the measure $\nu_{v,\beta}$ has no atom.
\end{itemize}
Suppose that $\beta > 0$. Then, the atom $x_{v,\beta}$ coincides with $\exp(F(\beta))$ and by Remark \ref{remark:atom} the mass of $\nu_{v,\beta}$ at $x_{v,\beta}$ is equal to $F'(\beta) = \partial_\beta ( x_{v,\beta} ) / x_{v,\beta} $, which yields the expression of $c_{v,\beta}$.
\end{proof}

Then, a straightforward consequence of Proposition \ref{th:ExplicitComp} is the following explicit formula for the free energy of the model, defined in \eqref{eq:defnFreeEn}. 
\begin{corollary}
The following equality holds:
\begin{equation}\label{eq:FreeEnergy}
F_\nu(\beta) =
\left\{
\begin{array}{lr}
0 & \text{if $\beta \leq 0$}, \\
\log\left( \frac{1}{1-\gamma_{v,\beta}} \right) & \text{if $\beta> 0$}.
\end{array}
\right.
\end{equation}
Moreover, when $\beta>0$, as $N \rightarrow + \infty$,
\begin{equation}\label{eq:AsymptoticPartition}
Z_{N,\beta,v} \sim   \frac{\exp(-\beta)}{1-v} \frac{\gamma_{v,\beta}^v}{1-\gamma_{v,\beta}} \left(\frac{1}{1-\gamma_{v,\beta}} \right)^N.
\end{equation}
When $\beta = 0$,

\[  Z_{N,\beta} = \frac{\sin( \pi v)}{\pi} \int_0^1 x^{N-v} (1-x)^{v-1} \mathrm{d}x  =  \frac{\Gamma(N+1-v) }{\Gamma(1-v) \Gamma(N+1)} \sim \frac{N^{-v}}{\Gamma(1-v)}. \]

\end{corollary}
\bigskip

\section{An epilogue on random matrix theory}
We end this article by presenting a link between the phase transition of a particular instance of the above pinned model and a famous phase transition in random matrix theory. Since it was the starting point of the present work, we think it justifies its presence here, at least in our view.

When $v=1/2$, $\mu_v$ is the classical Arcsine law:
\[ \mathrm{d} \mu_{\frac{1}{2}}(x) = \frac{1}{\pi} \frac{1}{\sqrt{(1-x)x}} \mathbf{1}_{x \in (0,1)} \mathrm{d}x. \]
In that case,
\[ \forall n \geq 1, \quad K_{\frac{1}{2}}(n) = \frac{1}{2^{2n}} \frac{1}{2n - 1} \binom{2n}{n}, \]
which is also the probability that the first return to $0$ of the simple random walk is equal to $2n$, see for example \cite{FellerVol1}. It turns out that in this setting, the phase transition from the delocalized regime to the localized regime for the polymer model corresponds to a famous phase transition in random matrix theory, which we briefly recall in the following. 

For all $n \geq 1$, let $X_n$ be a matrix of size $n \times n$ whose entries are i.i.d. random variables, centered and reduced. Let also $\Sigma_n = \mathrm{Diag}(2 \e^\beta ,1, \ldots,1)$, where $\beta \in \mathbf{R}$. We consider the following random covariance matrix:
\begin{equation*}  
S_n = \frac{1}{4n} \Sigma_n^{1/2} X_n X_n^T \Sigma_n^{1/2}. 
\end{equation*}
Denoting $\lambda_1 \geq \cdots \geq \lambda_n$ the eigenvalues of $S_n$, it turns out that, in probability, the empirical spectral measure 
$ \frac{1}{n} \sum_{1 \leq i \leq n} \delta_{\lambda_i}$ 
weakly converges to the so-called Marchenko-Pastur law with parameter $1$, given by the density $(2/\pi)(1-x)^{1/2} x^{-1/2} \mathbf{1}_{0 < x <1} \mathrm{d}x$. Moreover, the possible existence of an eigenvalue outside the limiting support $(0,1)$, often called an {\it outlier}, is the object of the following phase transition. We will denote by $\phi_1$ the normalized eigenvector associated to $\lambda_1$.

\begin{theorem*}[Baïk, Ben Arous, Péché phase transition]\label{theo:BBP}
Let $e_1$ be the first vector of the canonical basis. Then, the following converges hold in probability:
\[ \lambda_1 \underset{n \rightarrow +\infty}{\longrightarrow} 
\left\{
\begin{array}{lr}
1 & \text{if $\beta \leq 0$}, \\
\frac{\e^{2\beta}}{2 \e^\beta -1}  &  \text{otherwise},
\end{array}
\right. \quad  \quad \text{and} \quad \quad
|\langle \phi_1, e_1 \rangle \rangle |^2 \underset{n \rightarrow +\infty}{\longrightarrow} 
\left\{
\begin{array}{lr}
0 & \text{if $\beta \leq 0$}, \\
\frac{2 \e^\beta -2}{\e^\beta -1} &  \text{otherwise}.
\end{array}
\right.   \]
\end{theorem*}
This result was first proved by Ba\"ik, Ben Arous and P\'ech\'e \cite{BBP2005} in a Gaussian setting. Another approach to this problem is to study the {\it spectral measure in direction $e_1$} -- see \cite{noiry2020spectral}, defined by
\[  \mu_{(S_n,e_1)} := \sum\limits_{i=1}^n | \langle \phi_i, e_1 \rangle |^2 \delta_{\lambda_i},   \]
where $\phi_i$ is the normalized eigenvector associated to $\lambda_i$. With our notations, it turns out that in probability, $\mu_{(S_n,e_1)}$ weakly converges to $\nu_{\frac{1}{2},\beta}$, which is given by
\[
\mathrm{d} \nu_{\frac{1}{2},\beta}(x) = \frac{\e^\beta}{ \pi x} \frac{\sqrt{(1-x)x}}{x(1- 2 \e^\beta ) + \e^{2 \beta}} \mathbf{1}_{0<x<1} \mathrm{d}x + \frac{2 \e^\beta -2}{2 \e^\beta -1} \mathrm{1}_{\beta > 0} \mathrm{d} \delta_{\frac{\e^{2 \beta}}{2 \e^\beta - 1}}(x).
\]
In particular, the atomic part of $\nu_{\frac{1}{2},\beta}$ allows to retrieve the convergences of Theorem \ref{theo:BBP}. Interestingly, this links the Ba\"ik, Ben Arous and Péché phase transition for the largest eigenvalue of deformed random covariance matrices to the phase transition from the delocalized to the localized regime for the polymer model. In the super-critical regimes, the limit of $\log \lambda_1$ is the free energy of the polymer, and the limit of the square projection of the associated eigenvector is the multiplicative factor in front of the exponential term of the partition function -- this can be seen from Equations \eqref{eq:FreeEnergy} and \eqref{eq:AsymptoticPartition}.

\bigskip
{\bf Acknowledgements:} It is a pleasure for the authors to thank Quentin Berger for giving a lot of his time for sharing his deep knowledge on random polymers. 
The authors want also to thank the grants GDR MEGA and ANR MALIN for their financial support. Finally, we want to cheerfully thank the referees for their kind interest in our work and their constructive and insightful comments, which contributed to a substantial improvement of the present paper.


\bigskip

\noindent \textsc{Nathana\"el Enriquez:} nathanael.enriquez@universite-paris-saclay.fr

\noindent \textsc{Institut Math\'ematiques d'Orsay, B\^atiment 307, Universit\'e  Paris-Saclay, 91405 Orsay France}

\medskip

\noindent \textsc{Nathan Noiry:} \verb|noirynathan@gmail.com|

\noindent \textsc{Modal'X, UPL, Univ. Paris Nanterre, F92000 Nanterre France}

\end{document}